\documentclass[a4paper,12pt]{article}

\usepackage{amssymb}

\usepackage{booktabs}
\usepackage{amsmath}
\usepackage{geometry}
\usepackage{bm}

\usepackage{algorithmicx}
\usepackage[ruled]{algorithm}
\usepackage{algpseudocode}

\usepackage{latexsym}
\usepackage{mathrsfs}
\usepackage{amsfonts,amsmath,amssymb}
\usepackage{indentfirst}
\usepackage{subeqnarray}
\usepackage{verbatim}
\usepackage[pdftex]{graphicx}
\usepackage{epstopdf}

\newcommand{\bx}{\mbox{\boldmath{$x$}}}

\newcommand{\btau}{\mbox{\boldmath{$\tau$}}}
\newcommand{\bzero}{\mbox{\boldmath{$0$}}}

\newcommand{\fb}{\mbox{\boldmath{$f$}}}
\newcommand{\bn}{\mbox{\boldmath{$n$}}}
\newcommand{\bu}{\mbox{\boldmath{$u$}}}

\newcommand{\bv}{\mbox{\boldmath{$v$}}}

\newcommand{\bvarepsilon}{\mbox{\boldmath{$\varepsilon$}}}
\newcommand{\bsigma}{\mbox{\boldmath{$\sigma$}}}

\newcommand{\bnu}{\mbox{\boldmath{$\nu$}}}

\newtheorem{theorem}{Theorem}
\newtheorem{lemma}[theorem]{Lemma}

\newtheorem{problem}[theorem]{Problem}

\numberwithin{equation}{section}
\newenvironment{proof}[1][Proof]{\textbf{#1.} }
{\ \rule{0.75em}{0.75em}\smallskip}

\usepackage[pdftex,unicode,colorlinks,linkcolor=blue]{hyperref}

\begin{document}

\begin{center}
\Large\bf Numerical Analysis of a Contact Problem with Wear
\end{center}

\begin{center}
Danfu Han\footnote{Department of Mathematics, Hangzhou Normal University, Hangzhou, China. 
Email: {\tt mhdf@hznu.edu.cn}},\quad 
Weimin Han\footnote{Program in Applied Mathematical and Computational Sciences (AMCS) \&
Department of Mathematics, University of Iowa, Iowa City, IA 52242, USA.
Email: {\tt weimin-han@uiowa.edu}},\quad 
Michal Jureczka\footnote{Jagiellonian University in Krakow, Faculty of Mathematics and Computer Science, 
Lojasiewicza 6, 30-348 Krakow, Poland. Email: {\tt michal.jureczka@uj.edu.pl}}\quad and \quad 
Anna Ochal\footnote{Jagiellonian University in Krakow, Faculty of Mathematics and Computer Science, 
Lojasiewicza 6, 30-348 Krakow, Poland. Email: {\tt ochal@ii.uj.edu.pl}}
\end{center}

\begin{quote}
{\bf Abstract}. This paper represents a sequel to \cite{JO18} where numerical solution of a quasistatic 
contact problem is considered for an elastic body in frictional contact with a
moving foundation. The model takes into account wear of the contact surface of the body
caused by the friction. Some preliminary error analysis for a fully discrete approximation of the
contact problem was provided in \cite{JO18}.  In this paper, we consider a more general fully discrete
numerical scheme for the contact problem, derive optimal order error bounds and present computer
simulation results showing that the numerical convergence orders match the theoretical predictions.


\end{quote}

\section{Introduction}
\setcounter{equation}0

Contact phenomenon is common in engineering applications.  Mathematical studies and numerical analysis 
of contact problems are most suitably carried out within the frameworks of variational inequalities or 
hemivariational inequalities, which have attracted the attention of many researchers.  The related 
mathematical literature grows rapidly.  Some representative comprehensive references in this area are
\cite{DL1976, KO1988, HHNL1988, HS2002, SST2004, EJK2005, SHS2006, HR2013} in the context of variational 
inequalities and \cite{P1993, NP1995, HMP1999, MOS2013, SM2018} in the context of hemivariational inequalities.

For a contact problem, the mathematical model is constructed based on considerations of various aspects 
of the contact process.  Factors to be taken into account include the type of the contact process
(static, quasi-static, or dynamic), constitutive relations of the deformable bodies, contact 
conditions of various application-specific forms.  In certain applications, it is important
to consider heating or thermo effects (\cite{OJ18}), or piezoelectricity effects (\cite{Sz17}).  
Since the contact process inevitably causes material wear or even damage, it is not surprising that 
the wear effect has been built into mathematical models for a variety of contact processes, cf.\ 
\cite{CHS00, HSS01, RTS01, KS02, Bar06, VRS13, GOS16}.  In a recent paper \cite{SPS17}, a mathematical 
model is proposed and studied for contact with wear described by Archard's law of surface wear.
In this model, the friction between a deformable body and the foundation leads to wear of the 
contact surface of the body over time. Solution existence and uniqueness for the model are
provided in \cite{SPS17}.  Numerical approximation of the contact problem is the subject of \cite{JO18}
where some error bounds are derived for a fully discrete scheme.  In this paper, we take a further
step by considering a more general fully discrete numerical scheme for the contact problem 
that allows an arbitrary partition of the time interval, 
providing optimal order error estimates of the fully discrete scheme to solve the contact problem.
Moreover, we present numerical results showing deformation of the contact body and numerical convergence 
orders of the fully discrete solutions that confirm the theoretical error bounds.

The remainder of this paper is organized as follows. In Section \ref{sec2} we introduce the contact problem
and its variational formulation. In Section \ref{sec4}, we study a fully discrete numerical scheme
and derive optimal order error bounds. In Section \ref{sec5}, we present computational simulation results 
for numerical convergence orders that match the theoretical predictions.

\section{The contact problem and its variational formulation}\label{sec2}
\setcounter{equation}0

First, we describe the physical setting of the contact problem.  Consider a deformable body 
that occupies a domain $\Omega\subset\mathbb{R}^d$, $d = 2, 3$ in application. 
The body is subject to the action of volume forces with a total density $\fb_0$.
The boundary $\Gamma$ of the domain $\Omega$ is assumed to be Lipschitz continuous and is divided into
three disjoint measurable parts $\Gamma_D$, $\Gamma_N$ and $\Gamma_C$, with ${\rm meas}\,(\Gamma_D)>0$. 
Denote by $\bnu$ the unit outward normal vector on $\Gamma$ that is defined a.e.\ on $\Gamma$. 
The body is clamped on $\Gamma_D$, i.e., the displacement is equal to $\bzero$ on $\Gamma_D$. 
Surface transactions of a total density $\fb_N$ act on the boundary $\Gamma_N$. 
The contact boundary is $\Gamma_C$ where the contact is modeled by a normal compliance condition 
with a unilateral constraint and Coulomb's law of dry friction. Following \cite{SPS17}, we assume that 
the body is elastic, in contact with a moving obstacle (foundation) made of a hard perfectly rigid material, 
and assume that the contact surface of the body $\Gamma_C$ is covered by a layer of soft material. 
This layer is deformable and the foundation may penetrate it. Frictional contact with the foundation
may cause this layer to wear over time.

We assume that the acceleration of the body is negligible and so the problem is quasistatic.  In 
our model, the framework of the small strain theory is employed. We are interested in the body displacement 
and foundation wear in a time interval $[0, T]$, with $T > 0$.  We denote by ``$\cdot$'' and $\|\cdot\|$ 
the scalar product and the Euclidean norm in $\mathbb{R}^d$ or $\mathbb{S}^d$, respectively, where 
$\mathbb{S}^d$ is the space of symmetric matrices of order $d$.  The indices $i$ and $j$ run from $1$ to $d$ 
and the index after a comma represents the partial derivative with respect to the corresponding component 
of the independent variable. Summation convention over repeated indices is adopted. We denote the 
divergence operator by ${\rm Div}\bsigma = (\sigma_{ij,j})$ for an $\mathbb{S}^d$-valued field $\bsigma$. 
Standard Lebesgue and Sobolev spaces will be used, such as $L^2(\Omega)^d = L^2(\Omega; \mathbb{R}^d)$
and $H^1(\Omega)^d =H^1(\Omega; \mathbb{R}^d)$.  Recall that the linearized strain tensor 
of a displacement field $\bu\in H^1(\Omega)^d$ is
\[\bvarepsilon(\bu)=(\varepsilon_{ij}(\bu)),\quad\varepsilon_{ij}(\bu)=\frac{1}{2}\left(u_{i,j}+u_{j,i}\right).\]

Let $u_\nu = \bu\cdot\bnu$ and $\sigma_\nu =\bsigma\bnu\cdot\bnu$ be the normal components of 
$\bu$ and $\bsigma$, respectively, and let $\bu_\tau = \bu-u_\nu \bnu$ and
$\bsigma_\tau= \bsigma\bnu-\sigma_\nu \bnu$ be their tangential components, respectively. 
To simplify the notation, we will usually not indicate explicitly the dependence
of various functions on the spatial variable $\bx$.

Denote by $\bv^*(t)\not=\bzero$ the velocity of the foundation. Let 
\begin{equation}
\bn^*(t)=-\bv^*(t)/\|\bv^*(t)\|, \quad \alpha(t)=\kappa\,\|\bv^*(t)\|, 
\label{bn}
\end{equation}
where $\kappa$ represents the wear coefficient, and 
let $\mu$ be the friction coefficient.  The classical formulation of the contact problem with wear
is as follows.

\begin{problem}\label{P1}
Find a displacement field $\bu \colon\Omega\times [0,T]\to \mathbb{R}^d$, a stress field 
$\bsigma\colon\Omega\times [0,T]\to \mathbb{S}^d$, and a wear function
$w\colon\Gamma_C\times[0, T] \to \mathbb{R}_{+}=[0,\infty)$ such that for all $t \in [0, T]$,
\begin{align}	
\bsigma(t)={\mathcal F}\bvarepsilon(\bu(t))\quad&{\rm in}\ \Omega,\label{eq1}\\[1mm]
{\rm Div}\,\bsigma(t)+\fb_0(t)=\bzero\quad&{\rm in}\ \Omega,\label{eq2}\\[1mm]
\bu(t)=\bzero\quad &{\rm on}\ \Gamma_D,\label{eq3}\\[1mm]
\bsigma(t)\bnu=\fb_N(t)\quad&{\rm on}\ \Gamma_N,\label{eq4}\\[1mm]
\left. \begin{array}{ll}
u_\nu(t)\le g,\ \sigma_\nu(t)+p(u_\nu(t)-w(t))\le 0,\\
(u_\nu(t)-g)(\sigma_\nu(t)+p(u_\nu(t)-w(t)))=0
\end{array}\right\}  \quad&{\rm on}\ \Gamma_C,\label{eq5}\\[1mm]
-\bsigma_\tau(t)=\mu\,p(u_\nu(t)-w(t))\,\bn^*(t)\quad&{\rm on}\ \Gamma_C,\label{eq6}\\
w^\prime(t)=\alpha(t)\,p(u_\nu(t)-w(t))\quad&{\rm on}\ \Gamma_C,\label{eq7}\\
w(0)=0\quad&{\rm on}\ \Gamma_C.  \label{eq8}
\end{align}
\end{problem}

In Problem \ref{P1}, equation \eqref{eq1} represents an elastic constitutive law with an elasticity operator
${\cal F}$. Equation \eqref{eq2} is the equilibrium equation. The equality \eqref{eq3} describes the fact 
that body is clamped on $\Gamma_D$ and \eqref{eq4} represents external forces acting on $\Gamma_N$. 
The relations in \eqref{eq5} describe the damping response of the foundation, $g > 0$ being the 
thickness of a soft layer covering $\Gamma_C$. The friction is modeled by equation \eqref{eq6}. 
Here, the size of $\bv^*$ is assumed to be significantly larger than that of the 
tangential body velocity $\bu^\prime_\tau$. Equations \eqref{eq7} and \eqref{eq8} govern the
evolution of the wear function. Detailed derivation of this model is presented in \cite{SPS17}.

The contact problem will be studied in its variational formulation.  For this purpose, we introduce
function spaces and hypotheses on the problem data.  We recall that for a normed space $X$, $C([0,T];X)$
is the space of continuous functions from $[0, T]$ to $X$. We will use the following Hilbert spaces:
\[ {\cal H} = L^2(\Omega;\mathbb{S}^d),\quad V=\{\bv\in H^1(\Omega)^d\mid \bv =\bzero\ {\rm on}\ \Gamma_D\}\]
endowed with the inner scalar products
\[ (\bsigma,\btau)_{\cal H}=\int_\Omega \sigma_{ij}\tau_{ij} dx,\quad 
 (\bu,\bv)_V=(\bvarepsilon(\bu),\bvarepsilon(\bv))_{\cal H} \]
with the corresponding norms. Denote by $\langle\cdot,\cdot\rangle_{V^*\times V}$ the
duality pairing between a dual space $V^*$ and $V$.  The set of admissible displacements is
\[ U = \{\bv \in V \mid v_\nu\le g \ {\rm on}\ \Gamma_C \}.\]

For a function $\bv\in V$, we use the same symbol $\bv$ for its trace on the boundary $\Gamma$. 
By the Sobolev trace theorem, there exists a constant $c_0 > 0$ depending only on $\Omega$,
$\Gamma_D$ and $\Gamma_C$ such that
\begin{equation}
\|\bv\|_{L^2(\Gamma_C)^d}\le c_0\|\bv\|_V\quad\forall\,\bv\in V.
\label{eq9}
\end{equation}

Now we introduce the hypotheses on the data needed in the study of Problem \ref{P1}.

\noindent\underline{$H({\cal F})$}: For the elasticity operator 
${\cal F}\colon\Omega\times\mathbb{S}^d\to\mathbb{S}^d$,\\
\phantom{a}\quad (a) ${\cal F}(\cdot,\bvarepsilon)$ is measurable on $\Omega$ for all 
$\bvarepsilon\in \mathbb{S}^d$, ${\cal F}(\cdot,\bzero)\in {\cal H}$;\\
\phantom{a}\quad (b) $\exists$ $L_{\cal F} > 0$ s.t.\ $\|{\cal F}(\bx,\bvarepsilon_1)-{\cal F}(\bx,\bvarepsilon_2)\|
\le L_{\cal F}\|\bvarepsilon_1- \bvarepsilon_2\|$ $\forall\,\bvarepsilon_1,\bvarepsilon_2\in \mathbb{S}^d$, 
a.e.\ $\bx\in\Omega$;\\
\phantom{a}\quad (c) $\exists$ $m_{\cal F} > 0$ s.t.\ $({\cal F}(\bx,\bvarepsilon_1)- {\cal F}(\bx,\bvarepsilon_2)) 
\cdot(\bvarepsilon_1- \bvarepsilon_2)\ge m_{\cal F}\|\bvarepsilon_1- \bvarepsilon_2\|^2$ 
$\forall\,\bvarepsilon_1,\bvarepsilon_2\in \mathbb{S}^d$, a.e.\ $\bx\in\Omega$.

\noindent\underline{$H(p)$}: For the normal compliance function $p\colon\Gamma_C\times\mathbb{R}\to\mathbb{R}_+$,\\
\phantom{a}\quad (a) $p(\cdot, r)$ is measurable on $\Gamma_C$ $\forall\,r \in\mathbb{R}$;\\
\phantom{a}\quad (b) $\exists$ $L_p>0$ s.t.\ $|p(\bx,r_1)-p(\bx,r_2)|\le L_p|r_1-r_2|$ $\forall\,r_1,r_2\in \mathbb{R}$,
a.e.\ $\bx\in\Gamma_C$;\\
\phantom{a}\quad (c) $(p(\bx, r_1)-p(\bx, r_2))(r_1-r_2)\ge 0$ $\forall\,r_1, r_2\in \mathbb{R}$, a.e.\ $\bx\in\Gamma_C$;\\
\phantom{a}\quad (d) $p(\bx, r) = 0$ $\forall\, r\le 0$, a.e.\ $\bx\in\Gamma_C$.

\smallskip
Note that \underline{$H(p)$}\,(b) and (d) imply
\begin{equation}
|p(\bx,r)|\le L_p|r|\quad \forall\,r\in\mathbb{R},\ {\rm a.e.}\ \bx\in\Gamma_C.
\label{eq2.9a}
\end{equation}

\noindent\underline{$H(\fb)$}: For the densities of body and traction forces,
\[ \fb_0 \in C([0, T]; L^2(\Omega)^d ),\quad \fb_N\in C([0, T]; L^2(\Gamma_N)^d).\]

\noindent\underline{$H_0$}: For the friction and wear coefficients, and the foundation velocity,\\
\phantom{a}\quad (a) $\mu\in L^\infty(\Gamma_C)$, $\mu(\bx)\ge 0$ a.e.\ $\bx\in \Gamma_C$;\\
\phantom{a}\quad (b) $\kappa \in L^\infty(\Gamma_C)$, $\kappa(\bx)\ge 0$ a.e.\ $\bx\in \Gamma_C$;\\
\phantom{a}\quad (c) $\bv^* \in C([0, T ];\mathbb{R}^d)$, $\|\bv^*(t)\|\ge v_0 > 0$ $\forall\, t \in [0, T]$.

We notice that hypotheses \underline{$H_0$} implies the following regularities:
\begin{equation}
\bn^*\in C([0, T];\mathbb{R}^d),\quad \alpha\in C([0, T ]; L^\infty(\Gamma_C)),
\label{eq10}
\end{equation}
where $\bn^*$ and $\alpha$ are defined in \eqref{bn}.
 
Finally, we will need a smallness assumption on the combined effect of the Lipschitz constant of the 
normal compliance function $p$ and the friction coefficient $\mu$.  Recall that $c_0$ is the constant 
in the inequality \eqref{eq9}.

\noindent\underline{$H_s$}: $c^2_0 L_p\|\mu\|_{L^\infty(\Gamma_C)}< m_{\cal F}$.

Now we define some operators and functions needed in the variational formulation of Problem~\ref{P1}.
Let $F\colon V\to V^*$, $\fb\colon[0, T ]\to V^*$ and $\varphi\colon [0,T]\times L^2(\Gamma_C)\times
V \times V\to \mathbb{R}$ be defined for all $\bu,\bv\in V$, $w \in L^2(\Gamma_C)$, $t\in[0,T]$ as follows:
\begin{align*}
&\langle F\bu,\bv\rangle_{V^*\times V}=({\cal F}(\bvarepsilon(\bu)),\bvarepsilon(\bv))_{\cal H},\\
&\langle \fb(t),\bv\rangle_{V^*\times V}=\int_\Omega\fb_0(t)\cdot\bv\,dx+\int_{\Gamma_N}\fb_N(t)\cdot\bv\,da,\\
&\varphi(t,w,\bu,\bv)=\int_{\Gamma_C} p(u_\nu-w) \left[v_\nu+\mu\,\bn^*(t)\cdot\bv_\tau\right] da.
\end{align*}

Let $W=L^2(\Gamma_C)$ be the space for the wear variable $w$.  Using the standard procedures
in the mathematical theory of contact mechanics, we obtain the week formulation of Problem~\ref{P1}.

\begin{problem}\label{P2}
Find $\bu \colon [0,T]\to U$ and $w\colon [0,T] \to W$ such that for all $t\in[0,T]$,
\begin{align}
& \langle F\bu(t),\bv-\bu(t)\rangle_{V^*\times V}+\varphi(t,w(t),\bu(t),\bv)-\varphi(t,w(t),\bu(t),\bu(t))
\nonumber\\
&\qquad \qquad\qquad {}\ge \langle \fb(t),\bv-\bu(t)\rangle_{V^*\times V}\quad\forall\,\bv\in U,\label{eq11}\\
& w(t)=\int_0^t \alpha(s)\,p(u_\nu(s)-w(s))\,ds.  \label{eq12}
\end{align}
\end{problem}

We recall the following existence and uniqueness result for Problem \ref{P2} from \cite{SPS17}.

\begin{theorem}\label{thm1}
Assume \underline{$H({\cal F})$}, \underline{$H(p)$}, \underline{$H(\fb)$}, \underline{$H_0$} 
and \underline{$H_s$}. Then Problem \ref{P2} has a unique solution 
with the regularity
\[ \bu\in C([0, T ]; V), \quad \bsigma\in C([0, T ];{\cal H}),\quad  w \in C^1([0,T];W). \]
In addition, $w(t)\ge 0$ for all $t\in [0,T]$, a.e.\ on $\Gamma_C$.
\end{theorem}

\section{Numerical analysis}\label{sec4}
\setcounter{equation}0

We turn to the numerical solution of Problem \ref{P2}.  
Let $V^h \subset V$ and $W^h\subset W$ be two families of finite dimensional subspaces with
a discretization parameter $h > 0$. Then define $U^h = U \cap V^h$.  Let $0=t_0<t_1<\cdots<t_N=T$
be a partition of the time interval $[0,T]$.  Denote $k_n=t_{n+1}-t_n$, $0\le n\le N-1$, and 
$k=\max_{0\le n\le N-1} k_n$ for the time step size.
For a function $z$ continuous in $t$, we write $z_n = z(t_n)$.

We make the following additional assumptions on the solution $\bu$ to Problem \ref{P2}
and the velocity of the foundation $\bv^*$.

\noindent\underline{$H_1$}: $\bu\in H^1(0, T ; V)$, $\bv^*\in W^{1,\infty}(0,T;\mathbb{R}^d)$.

Note that assumptions \underline{$H_1$} and \underline{$H_0$}\,(b) imply that
\begin{equation}
\alpha\in W^{1,\infty}(0,T;L^\infty(\Gamma_C)). 
\label{eq13}
\end{equation}

Consider the following fully discrete scheme for solving Problem \ref{P2}.

\begin{problem}\label{P3}
Find $\bu^{hk}=\{\bu^{hk}_n\}_{n=0}^N\subset U^h$ and $w^{hk}=\{w^{hk}_n\}_{n=0}^N\subset W^h$,
$w^{hk}_0=0$, such that for $0\le n\le N$,
\begin{align}
&\langle F\bu^{hk}_n,\bv^h-\bu^{hk}_n\rangle_{V^*\times V}+\varphi(t_n,w^{hk}_n,\bu^{hk}_n,\bv^h)
-\varphi(t_n,w^{hk}_n,\bu^{hk}_n,\bu^{hk}_n)\nonumber\\
&\qquad\qquad\qquad\ge \langle \fb_n,\bv^h-\bu^{hk}_n\rangle_{V^*\times V}\quad\forall\,\bv^h\in U^h,
\label{eq14}
\end{align}
and for $1\le n\le N$,
\begin{equation}
w^{hk}_n=\sum_{j=0}^{n-1} k_j \alpha_j p(u^{hk}_{j,\nu}-w^{hk}_j).  \label{eq15}
\end{equation}
\end{problem}

We remark that existence of a unique solution to Problem \ref{P3} follows from an application 
of discrete version of Theorem \ref{thm1}.  We also remark that the numerical scheme considered
in \cite{JO18} is a special case of Problem \ref{P3} where a uniform partition of the time
interval $[0,T]$ is used.  For a uniform partition of $[0,T]$ into $N$ equal size sub-intervals, 
we let $k = T/N$ be the time step and $t_n = n\,k$, $0\le n\le N$, the node points. 

We will make use of the following discrete Gronwall inequality (\cite[Lemma 7.25]{HS2002}).

\begin{lemma}\label{lem3}
Assume $\{g_n\}_{n=1}^N$ and $\{e_n\}_{n=1}^N$ are two sequences of non-negative numbers satisfying
\[ e_n\le c\,g_n+c\,\sum_{j=1}^{n-1} k_j e_j,\quad n=1,\dots,N. \]
Then
\[ e_n\le c\,\Big(g_n+\sum_{j=1}^{n-1} k_j g_j\Big),\quad n=1,\dots,N. \]
Therefore,
\[ \max_{1 \le n \le N} e_n \le c\,\max_{1 \le n \le N} g_n. \]
\end{lemma}

We have Ce\'{a}'s inequality useful for error estimation.

\begin{theorem}\label{thm4}
Under the assumptions stated in Theorem \ref{thm1} and the additional hypothesis \underline{$H_1$}, there 
exists a constant $c > 0$ such that for any $\bv^h_n\in U^h$, $1\le n\le N$,
\begin{align}
\max_{1\le n\le N}\left(\|\bu_n-\bu^{hk}_n\|_V^2+\|w_n-w^{hk}_n\|_W^2\right)
& \le c\,k^2+ c\,k\,\|\bu_0-\bu^{hk}_0\|_V^2\nonumber\\
&\quad{} +c\,\max_{1\le n\le N}\left(\|\bu_n-\bv^h_n\|_V^2+|R_n(w_n,\bu_n,\bv^h_n)|\right) 
\label{eq17}
\end{align}
where
\begin{align}
R_n(w_n,\bu_n,\bv^h_n)& =\langle F\bu_n,\bv^h_n-\bu_n\rangle_{V^*\times V}
+\varphi(t_n,w_n,\bu_n,\bv^h_n)-\varphi(t_n,w_n,\bu_n,\bu_n)\nonumber\\
&\quad{} -\langle \fb_n,\bv^h_n-\bu_n\rangle_{V^*\times V}.
\label{eq18}
\end{align}
\end{theorem}
\begin{proof}
By modifying the proof of Theorem 4 in \cite{JO18}, we can establish the inequality
\begin{align}
\|\bu_n-\bu^{hk}_n\|_V^2+\|w_n-w^{hk}_n\|_W^2 &\le c\,\|\bu_n-\bv^h_n\|_V^2+\left|R_n(w_n,\bu_n,\bv^h_n)\right|
+c\,k^2+c\,k\,\|\bu_0-\bu^{hk}_0\|_V^2\nonumber\\
&\quad{}+ c\sum_{j=1}^{n-1} k_j \left(\|\bu_j-\bu^{hk}_j\|_V^2+\|w_j-w^{hk}_j\|_W^2\right).
\label{eq28}
\end{align}
Applying Lemma \ref{lem3} on \eqref{eq28}, we get the inequality \eqref{eq17}. \hfill
\end{proof}

Note that from \underline{$H(p)$} and \underline{$H_0$}, we have (cf.\ \cite[(27)]{JO18}), for $t\in [0,T]$,
\begin{align}
&\varphi(t,w_1,\bu_1,\bv_2)+\varphi(t,w_2,\bu_2,\bv_1)-\varphi(t,w_1,\bu_1,\bv_1)-\varphi(t,w_2,\bu_2,\bv_2)
\nonumber\\
&\qquad{}\le L_p\left(c_0\|\bu_1-\bu_2\|_V+\|w_1-w_2\|_W\right)
\left(c_0\|\mu\|_{L^\infty(\Gamma_C)}\|\bv_1-\bv_2\|_V+\|w_1-w_2\|_W\right)\nonumber\\
&\qquad{}\qquad\quad\forall\,\bu_1,\bu_2,\bv_1,\bv_2\in V,\,w_1,w_2\in W.
\label{eq49}
\end{align}

The inequality \eqref{eq17} is the starting point for further error estimation.  
For simplicity, we assume $\Omega$ is a polygonal/polyhedral domain.  Then $\Gamma_D$, $\Gamma_N$ and 
$\Gamma_C$ can be expressed as unions of flat components (line segments for $d=2$ and polygons 
for $d=3$) that have pairwise disjoint interiors.  In particular, we write
$\overline{\Gamma_C}=\cup_{i=1}^{i_0}\Gamma_{C,i}$, where each component $\Gamma_{C,i}$ is a line segment 
if $d=2$ or a polygon if $d=3$.
Consider a regular family of finite element partitions $\{{\cal T}^h\}$ of the domain $\overline{\Omega}$
into triangular or tetrahedral elements such that if the intersection of one side/face of an
element with one flat component of the boundary has a positive relative measure, then the side/face 
lies entirely in that flat component.  Corresponding to ${\cal T}^h$, we define the linear element space 
\begin{equation}
V^h=\left\{\bv^h\in C(\overline{\Omega})^d \mid \bv^h|_T\in \mathbb{P}_1(T)^d,\ T\in {\cal T}^h,\
   \bv^h=\bzero\ {\rm on\ }\Gamma_D\right\}. 
\label{Vh}
\end{equation}
Then we define the discrete admissible finite element set
\begin{equation}
U^h=\left\{\bv^h\in V^h\mid v^h_\nu\le g\ {\rm at\ all\ nodes\ on\ }\Gamma_C\right\}. 
\label{Uh}
\end{equation}
We assume $g$ is a concave function. Then, $U^h=V^h\cap U\subset U$.  We proceed to derive an optimal 
order error estimate for the finite element solution defined by Problem \ref{P3}.

\begin{theorem}\label{thm5}
Keep the assumptions stated in Theorem \ref{thm4}.  Assume further the solution regularities
\begin{align}
&\bu\in C([0,T];H^2(\Omega)^d),\quad \bu|_{\Gamma_{C,i}}\in C([0,T];H^2(\Gamma_{C,i})^d),\quad 1\le i\le i_0,
\label{eq29}\\
& \bsigma\bnu|_{\Gamma}\in C([0,T];L^2(\Gamma)^d). \label{eq30}
\end{align}
Then we have the optimal order error estimate
\begin{equation}
\max_{1\le n\le N}\left(\|\bu_n-\bu^{hk}_n\|_V^2+\|w_n-w^{hk}_n\|_W^2\right)
\le c\left(k^2+h^2\right).  \label{eq31}
\end{equation}
\end{theorem}
\begin{proof}
By following the arguments presented in \cite[Section 8.1]{HS2002}, it can be shown that under 
the stated regularity assumptions, the solution of Problem \ref{P2} satisfies, for $t\in [0,T]$, 
\begin{align*}
{\rm Div}\,\bsigma(t)+\fb_0(t)=\bzero\quad&{\rm a.e.\ in}\ \Omega,\\[1mm]
\bsigma(t)\bnu=\fb_N(t)\quad&{\rm a.e.\ on}\ \Gamma_N,
\end{align*}
where
\[ \bsigma(t)={\mathcal F}\bvarepsilon(\bu(t)). \]
Using these relations  we find that
\[ R_n(w_n,\bu_n,\bv^h_n)=\int_{\Gamma_C}\left\{\bsigma_n\bnu{\cdot}(\bv^h_n-\bu_n)
+p(u^{hk}_{n,\nu}-w_n)\left[v^h_{n,\nu}-u_{n,\nu}
+\mu\,\bn^*_n\cdot(\bv_{n,\tau}^h-\bu_{n,\tau})\right]\right\}da. \]
Thus,
\begin{equation}
\left|R_n(w_n,\bu_n,\bv^h_n)\right|\le c\,\|\bu_n-\bv_n^h\|_{L^2(\Gamma_C)^d}. 
\label{eq32}
\end{equation}
This provides an upper bound for the term $\left|R_n(w_n,\bu_n,\bv^h_n)\right|$ on the right hand 
side of \eqref{eq17}.  

Now we bound the error $\|\bu_0-\bu^{hk}_0\|_V$.  For simplicity, we denote
\[ \varphi_0(\bu,\bv):=\varphi(0,0,\bu,\bv). \]
Write
\begin{align}
\langle F\bu_0-F\bu^{hk}_0,\bu_0-\bu^{hk}_0\rangle_{V^*\times V}
& = \langle F\bu_0-F\bu^{hk}_0,\bu_0-\bv^h_0\rangle_{V^*\times V}
 +\langle F\bu_0,\bv^h_0-\bu_0\rangle_{V^*\times V} \nonumber\\
&\quad{} +\langle F\bu_0,\bu_0-\bu^{hk}_0\rangle_{V^*\times V}
-\langle F\bu^{hk}_0,\bv^h_0-\bu^{hk}_0\rangle_{V^*\times V}.
\label{eq53}
\end{align}
From \eqref{eq11} with $t=0$,
\begin{equation}
\langle F\bu_0,\bv-\bu_0\rangle_{V^*\times V}+\varphi_0(\bu_0,\bv)-\varphi_0(\bu_0,\bu_0)
\ge \langle \fb_0,\bv-\bu_0\rangle_{V^*\times V}\quad\forall\,\bv\in U.
\label{eq51}
\end{equation}
From \eqref{eq14} with $n=0$,
\begin{equation}
\langle F\bu^{hk}_0,\bv^h_0-\bu^{hk}_0\rangle_{V^*\times V}+\varphi_0(\bu^{hk}_0,\bv^h_0)
-\varphi_0(\bu^{hk}_0,\bu^{hk}_0)\ge \langle \fb_0,\bv^h_0-\bu^{hk}_0\rangle_{V^*\times V}
\quad\forall\,\bv^h_0\in U^h.
\label{eq52}
\end{equation}
Take $\bv=\bu^{hk}_0$ in \eqref{eq51}, and use the resulting inequality and the inequality \eqref{eq52} 
in \eqref{eq53} to obtain
\begin{align}
\langle F\bu_0-F\bu^{hk}_0,\bu_0-\bu^{hk}_0\rangle_{V^*\times V}
& \le \langle F\bu_0-F\bu^{hk}_0,\bu_0-\bv^h_0\rangle_{V^*\times V}+R_0(0,\bu_0,\bv^h_0)\nonumber\\
&\quad{} +\varphi_0(\bu_0,\bu^{hk}_0)+\varphi_0(\bu^{hk}_0,\bv_0^h)-\varphi_0(\bu_0,\bv_0^h)
-\varphi_0(\bu^{hk}_0,\bu^{hk}_0).
\label{eq54}
\end{align}
By \underline{$H({\cal F})$}\,(c),
\[ m_{\cal F}\|\bu_0-\bu^{hk}_0\|_V^2\le \langle F\bu_0-F\bu^{hk}_0,\bu_0-\bu^{hk}_0\rangle_{V^*\times V}. \]
By \underline{$H({\cal F})$}\,(b),
\[  \langle F\bu_0-F\bu^{hk}_0,\bu_0-\bv^h_0\rangle_{V^*\times V} 
\le L_{\cal F}\|\bu_0-\bu^{hk}_0\|_V\|\bu_0-\bv^h_0\|_V.  \]
Then, for an arbitrarily small $\epsilon>0$, there is a constant $c$ depending on $\epsilon$ such that
\[  \langle F\bu_0-F\bu^{hk}_0,\bu_0-\bv^h_0\rangle_{V^*\times V} 
\le \epsilon\,\|\bu_0-\bu^{hk}_0\|_V^2+c\,\|\bu_0-\bv^h_0\|_V^2.  \]
By \eqref{eq32}, 
\[ R_0(0,\bu_0,\bv^h_0)\le c\,\|\bu_0-\bv^h_0\|_{L^2(\Gamma_C)^d}. \]
By \eqref{eq49},
\begin{align*} 
& \varphi_0(\bu_0,\bu^{hk}_0)+\varphi_0(\bu^{hk}_0,\bv_0^h)-\varphi_0(\bu_0,\bv_0^h)
  -\varphi_0(\bu^{hk}_0,\bu^{hk}_0)\\
&\qquad{}\le c^2_0 L_p\|\mu\|_{L^\infty(\Gamma_C)}\|\bu_0-\bu^{hk}_0\|_V\|\bu^{hk}_0-\bv^h_0\|_V.
\end{align*}
Since 
\[ \|\bu^{hk}_0-\bv^h_0\|_V\le \|\bu_0-\bu^{hk}_0\|_V+ \|\bu_0-\bv^h_0\|_V, \]
for the arbitrarily small $\epsilon>0$, there is a constant $c$ depending on $\epsilon$ such that
\begin{align*} 
& \varphi_0(\bu_0,\bu^{hk}_0)+\varphi_0(\bu^{hk}_0,\bv_0^h)-\varphi_0(\bu_0,\bv_0^h)
  -\varphi_0(\bu^{hk}_0,\bu^{hk}_0)\\
&\qquad{}\le \left(c^2_0 L_p\|\mu\|_{L^\infty(\Gamma_C)}+\epsilon\right)\|\bu_0-\bu^{hk}_0\|_V^2
  +c\,\|\bu_0-\bv^h_0\|_V^2.
\end{align*}
Using these relations in \eqref{eq54}, we obtain
\[ \left(m_{\cal F}-c^2_0 L_p\|\mu\|_{L^\infty(\Gamma_C)}-2\,\epsilon\right)\|\bu_0-\bu^{hk}_0\|_V^2\le
 c\left(\|\bu_0-\bv^h_0\|_V^2+\|\bu_0-\bv^h_0\|_{L^2(\Gamma_C)^d}\right). \]
Recall the condition \underline{$H_s$}; choosing $\epsilon=\left(m_{\cal F}-c^2_0 L_p\|\mu\|_{L^\infty(\Gamma_C)}\right)/4$
we obtain from the above inequality that
\begin{equation}
\|\bu_0-\bu^{hk}_0\|_V^2\le c\left(\|\bu_0-\bv^h_0\|_V^2+\|\bu_0-\bv^h_0\|_{L^2(\Gamma_C)^d}\right).
\label{eq55}
\end{equation}
Using \eqref{eq55} and \eqref{eq32} in \eqref{eq17}, we have
\begin{align}
\max_{0\le n\le N}\left(\|\bu_n-\bu^{hk}_n\|_V^2+\|w_n-w^{hk}_n\|_W^2\right)
& \le c\,k^2+ c\,k\left(\|\bu_0-\bv^h_0\|_V^2+\|\bu_0-\bv^h_0\|_{L^2(\Gamma_C)^d}\right)\nonumber\\
&\quad{} +c\,\max_{1\le n\le N}\left(\|\bu_n-\bv^h_n\|_V^2+\|\bu_n-\bv_n^h\|_{L^2(\Gamma_C)^d}\right)
\label{eq56}
\end{align}
for any $\bv^h_n\in U^h$.  

Thus, by applying the finite element interpolation theory (e.g., \cite{AH2009, Ciar1978}), 
we have the optimal order error bound \eqref{eq31} from \eqref{eq56},
under the solution regularities \eqref{eq29} and \eqref{eq30}.  \hfill
\end{proof}

We comment that if ${\cal F}(\bx,\bvarepsilon)$ is a smooth function of $\bx$, in particular if
${\cal F}(\bx,\bvarepsilon)$ does not depend on $\bx$, then \eqref{eq30} follows from \eqref{eq29}
and thus there is no need to assume \eqref{eq30}.

\section{Numerical results}\label{sec5}
\setcounter{equation}0

In this section, we report computer simulation results on a numerical example.  Let $d=2$ and 
consider a square-shaped set $\Omega=(0,1)\times(0,1)$ with the following partition of the boundary
\[ \Gamma_D=\{0\}\times[0,1],\quad\Gamma_N=([0,1]\times\{1\})\cup(\{1\}\times[0,1]),
\quad\Gamma_C=[0,1]\times\{0\}.\]

The linear elasticity operator $\mathcal{F}$ is defined by
\[\mathcal{F}(\bm{\tau}) = 2\eta\bm{\tau} + \lambda \mbox{tr}(\bm{\tau})I,\qquad \bm{\tau} \in \mathbb{S}^2.\]
Here $I$ denotes the identity matrix, $\mbox{tr}$ denotes the trace of the matrix, $\lambda>0$ and $\eta>0$ 
are the Lame coefficients.  In our simulations, we choose $\lambda = \eta = 4$, $T=1$ and take the following data
\begin{align*}
 &\bm{u}_{0}(\bm{x}) = (0,0), \quad \bm{x} \in \Omega,\\
 &p(r) = \left \{ \begin{array}{ll}
   100 \, r, \quad r \in [0, \infty), \\
   0, \quad r \in (-\infty, 0), \\
  \end{array} \right.\\
 &\bm{f}_N(\bm{x},t) = (0,0), \quad \bm{x} \in \Omega,\ t \in [0,T],\\
 &\bm{f}_0(\bm{x},t) = (0,-2), \quad \bm{x} \in \Omega,\ t \in [0,T],\\
 &g = 0.1.
\end{align*}

We use the linear finite element space $V^h$ defined in \eqref{Vh} and its subset $U^h$ 
defined in \eqref{Uh}, based on uniform triangular partitions of $\overline{\Omega}$. 
We use the uniform partition of the time interval $[0,1]$ with the 
time step size $k=1/N$ for a positive integer $N$.

We first demonstrate the effect of some input data on the deformation of the body.
In all cases, we show the shape of the body at final time $t=1$, and the numerical solutions correspond 
to the time step size $1/16$ and where the boundary $\Gamma_C$ of the body is divided into $16$ equal parts. 

In Figure \ref{figOne} we show the deformed configuration for $\mu(\bm{x}) = 0.3$, $\kappa(\bm{x})= 0.04$
and $\bm{v^*}(\bm{x},t)= (1,0)$. We push the body down towards the moving foundation with a force $\bm{f}_0$, 
and as a result of friction, the soft layer of material covering $\Gamma_C$ wears out allowing the body 
to move downward. We observe that in this case coefficient $\kappa$, governing the rate of wear, 
is not big enough to cause the body to touch the foundation. Because of the friction, the body moves 
in the same direction as the foundation, i.e.\ to the right.

We then increase the wear coefficient $\kappa$ to $\kappa(\bm{x}) = 0.08$.  The deformed configuration
is shown in Figure~\ref{figTwo}.  We observe that the layer of soft material on part of the boundary 
$\Gamma_C$ completely wears out, allowing the body to rest on the rigid foundation as it cannot penetrate
it further.

In Figure~\ref{figThree}, we show the deformed configuration for $\mu(\bm{x}) =1$, $\kappa(\bm{x})= 0.04$
and $\bm{v^*}(\bm{x},t)= (1,0)$.  We observe that the body moves further to the right, which is a result 
of increased friction between soft layer of material covering $\Gamma_C$ and the rigid foundation. 

The result in Figure~\ref{figFour} corresponds to $\mu(\bm{x}) = 0.3$, $\kappa(\bm{x})= 0.02$
and $\bm{v^*}(\bm{x},t)= (-1,0)$. Note that the direction of the motion of the foundation is reversed.
As a result, the lower part of the body squeezes to the left and we observe that the boundary $\Gamma_C$ 
is slightly curled. We conclude that all those modifications lead to results that can be expected.

\begin{figure}[ht]
\centering
\begin{minipage}{.45\textwidth}
  \centering
    \includegraphics[width=0.9\linewidth]{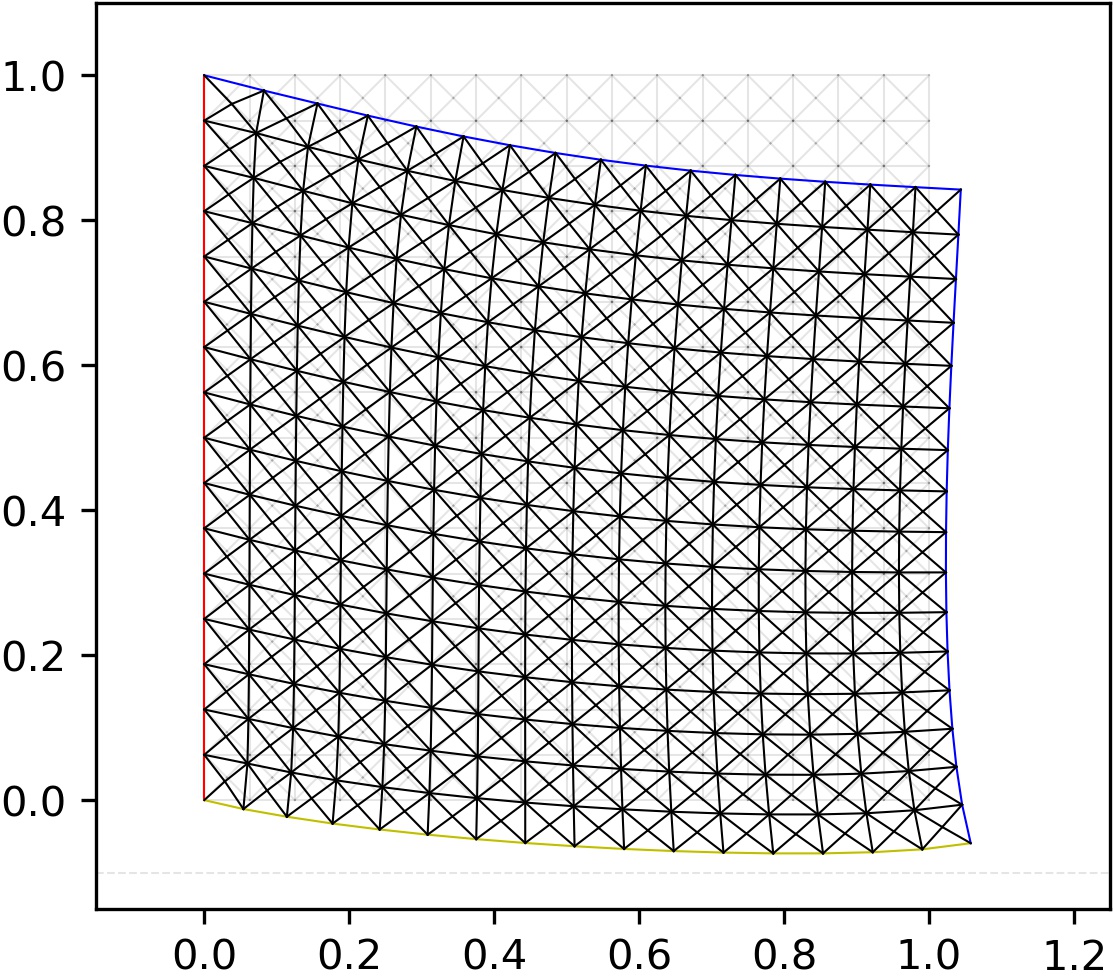}
    \caption{Deformed configuration at $t=1$,
    $\mu= 0.3$, $\kappa= 0.04$, $\bm{v^*}= (1,0)$} \label{figOne}
\end{minipage}
\begin{minipage}{.45\textwidth}
  \centering
    \includegraphics[width=0.9\linewidth]{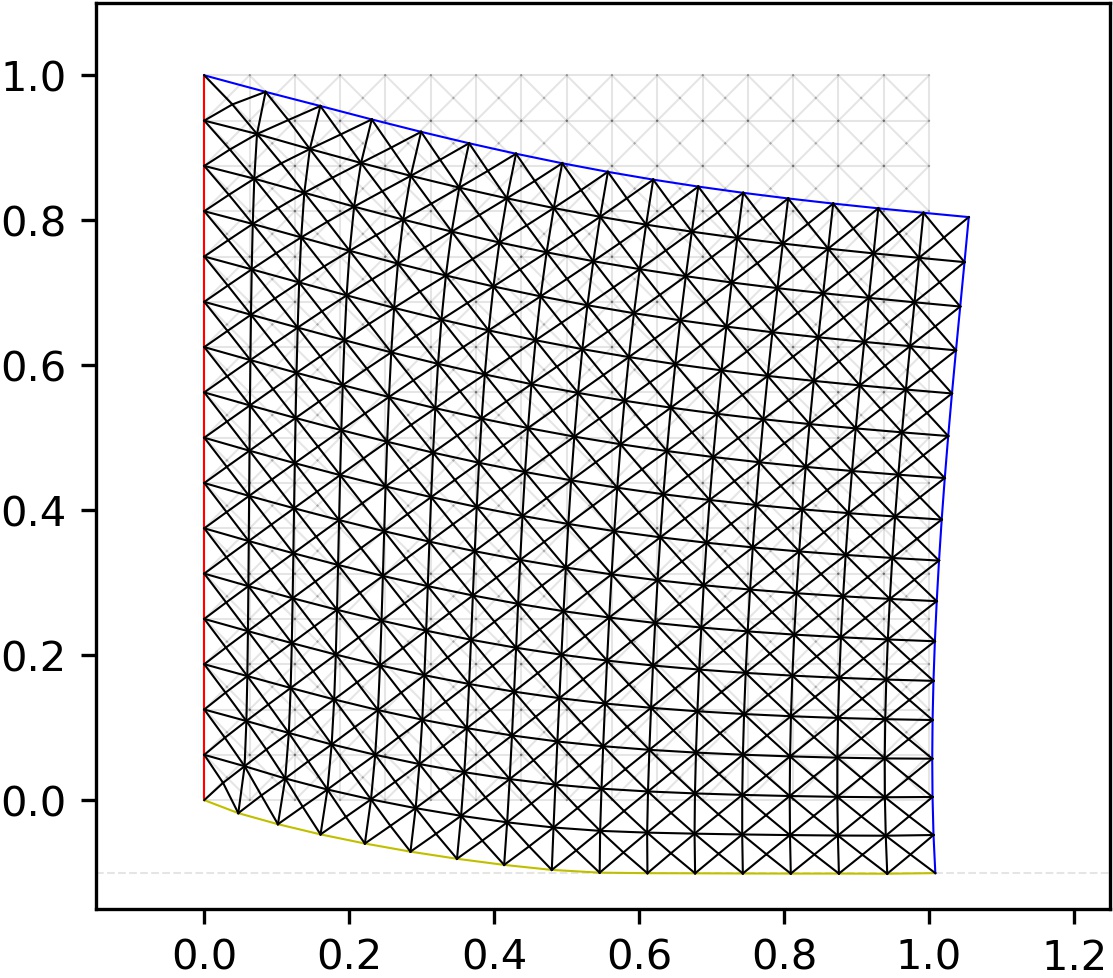}
    \caption{Deformed configuration at $t=1$,
    $\mu= 0.3$, $\kappa= 0.08$, $\bm{v^*}= (1,0)$} \label{figTwo}
\end{minipage}
\end{figure}

\begin{figure}[ht]
\centering
\begin{minipage}{.45\textwidth}
  \centering
 \includegraphics[width=0.9\linewidth]{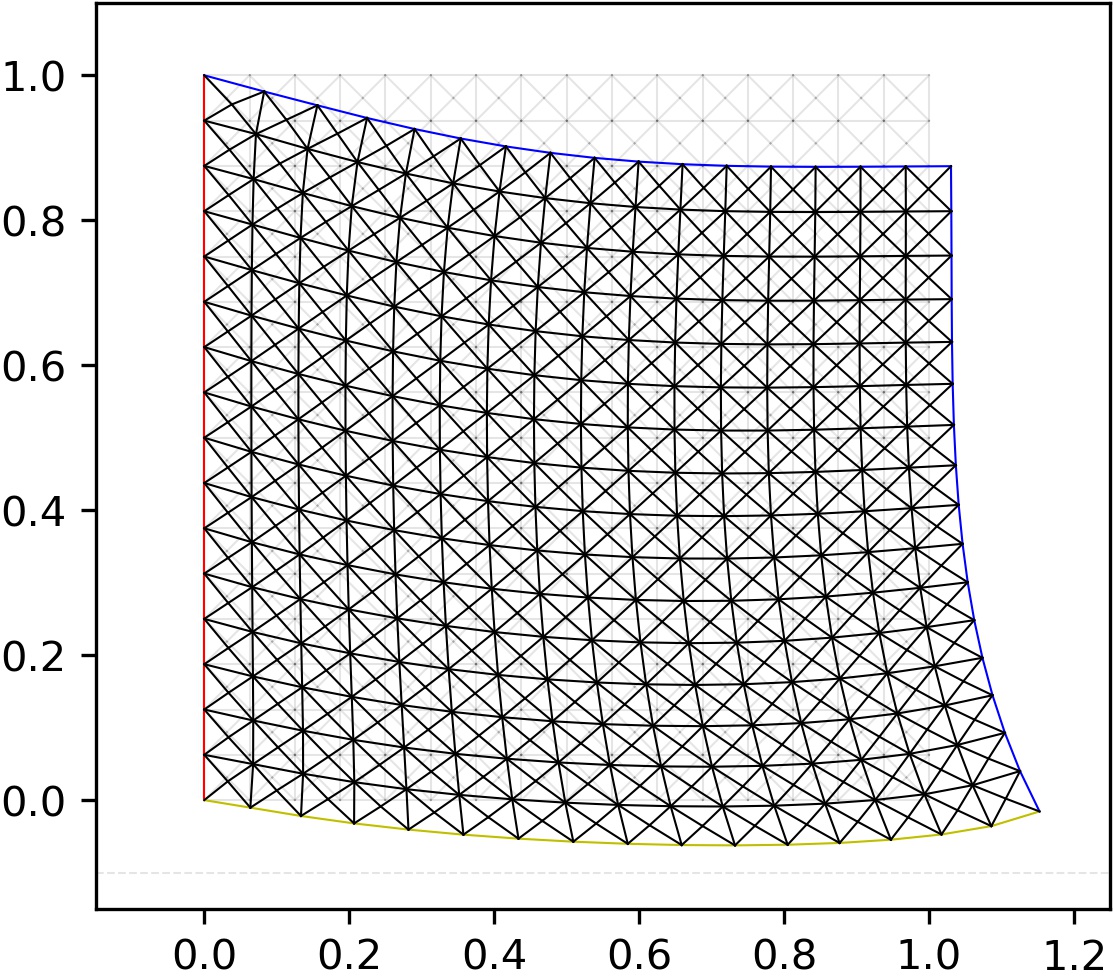}
    \caption{Deformed configuration at $t=1$,
    $\mu=1$, $\kappa= 0.04$, $\bm{v^*}= (1,0)$} \label{figThree}
\end{minipage}
\begin{minipage}{.45\textwidth}
  \centering
    \includegraphics[width=0.9\linewidth]{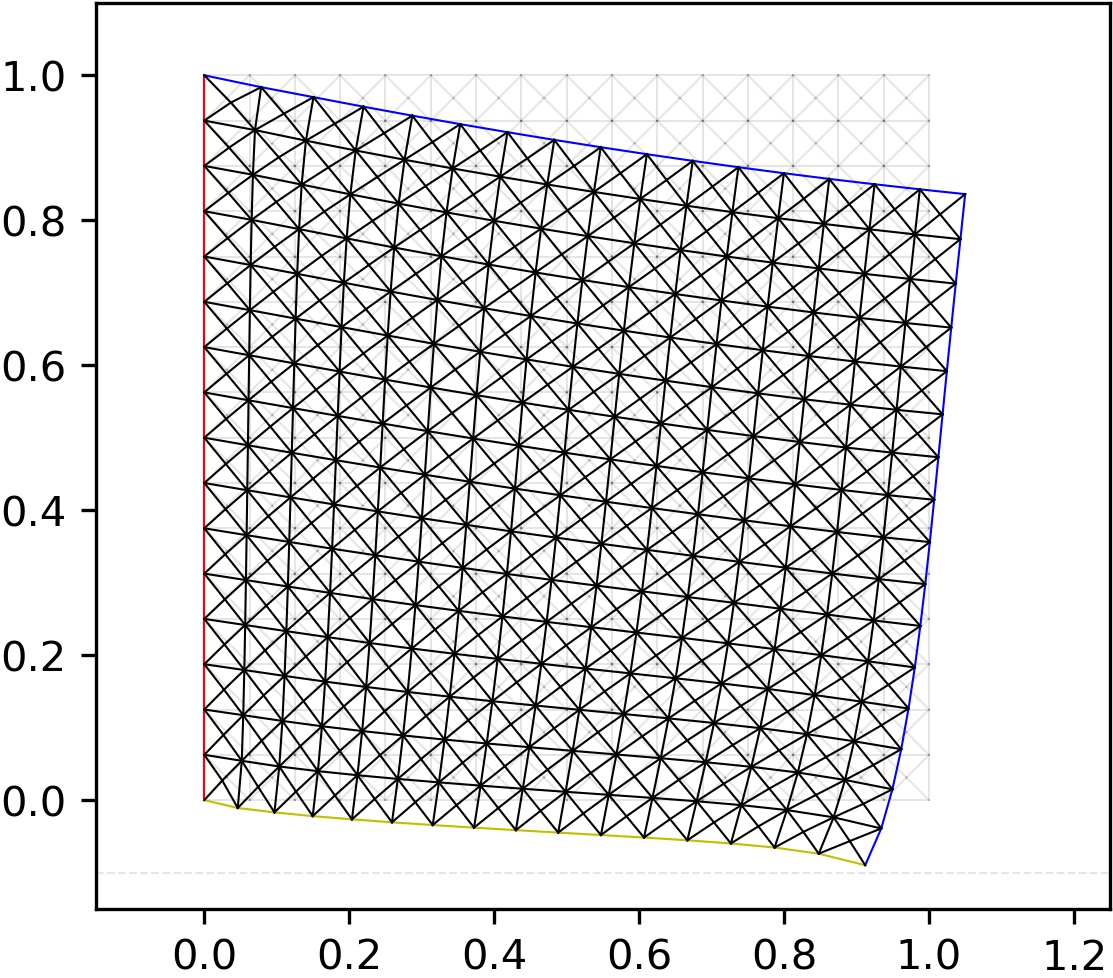}
    \caption{Deformed configuration at $t=1$,
    $\mu= 0.3$, $\kappa= 0.02$, $\bm{v^*}= (-1,0)$} \label{figFour}
\end{minipage}
\end{figure}

\begin{table}[ht]
\footnotesize
\centering
\begin{tabular}{ l r r r r r }\hline
$h+k$ & $1$ & $1/2$ & $1/4$ & $1/8$ & $1/16$ \\ \hline
$\|\bm{u}-\bm{u}^{hk}\|_V/\|\bm{u}\|_V$ & $4.1698e^{-1}$ & $2.6840e^{-1}$ & $1.4360e^{-1}$ 
&  $7.3979e^{-2}$ & $3.4882e^{-2}$ \\ 
{\rm Convergence order} & & 0.6355 &   0.9022 &   0.9569 &   1.0846 \\ \hline
$\|w-w^{hk}\|_{W}/\|w\|_W$&
$2.9009e^{-1}$&$1.0328e^{-1}$&$3.8385e^{-2}$&$1.4694e^{-2}$ & $5.0891e^{-3}$ \\ 
{\rm Convergence order} & &1.4898  &  1.4280   & 1.3853 &   1.5297 \\ \hline
\end{tabular}
\caption{Numerical errors} \label{tabOne}
\end{table}


\begin{figure}[ht]
\centering
\begin{minipage}{.45\textwidth}
  \centering
  \includegraphics[width=0.9\linewidth]{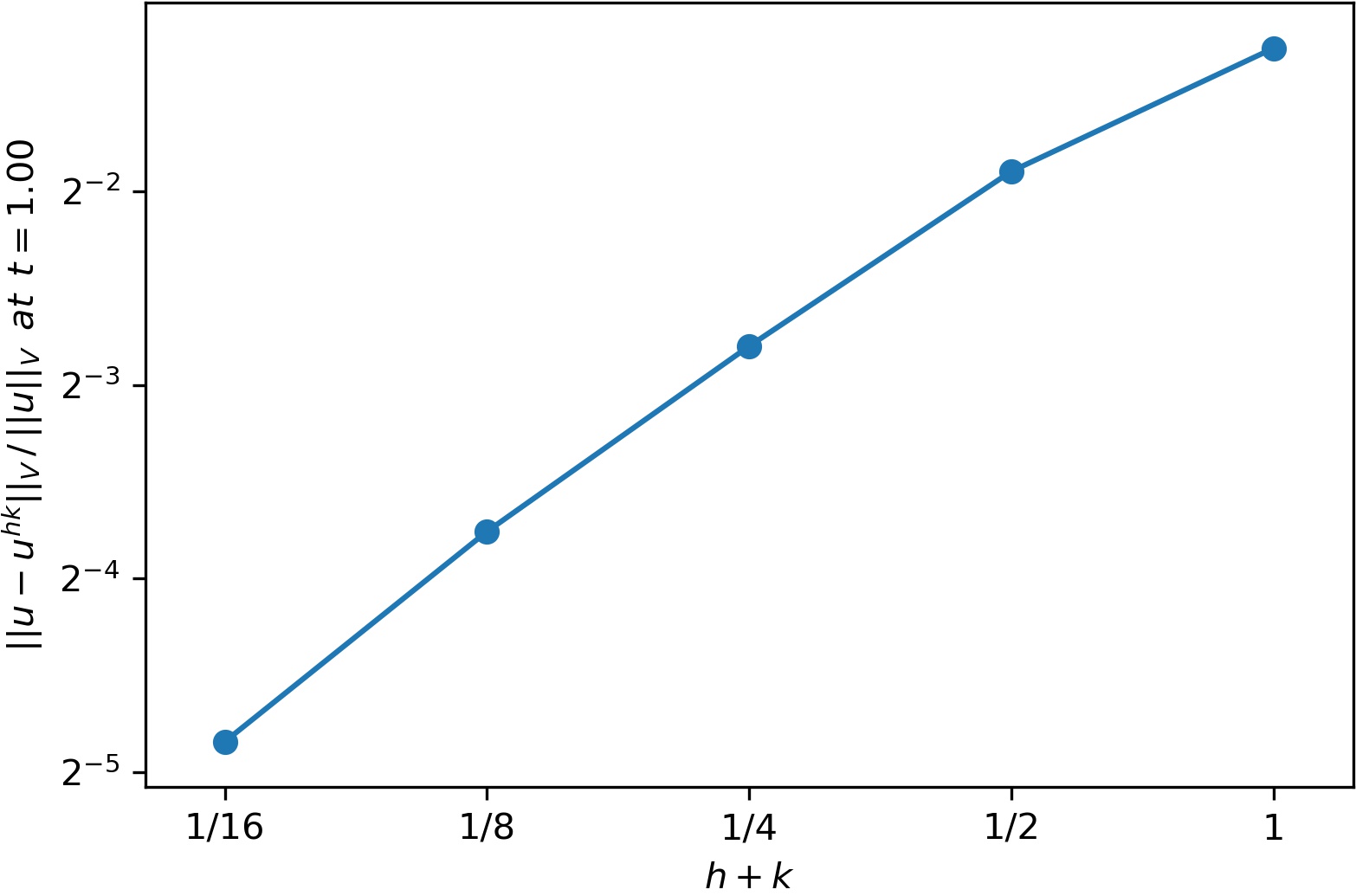}
  \caption{Error estimate $\|\bm{u}  - \bm{u}^{hk}\|_V / \|\bm{u}\|_V$} \label{figFive}
\end{minipage}
\begin{minipage}{.45\textwidth}
  \centering
  \includegraphics[width=0.9\linewidth]{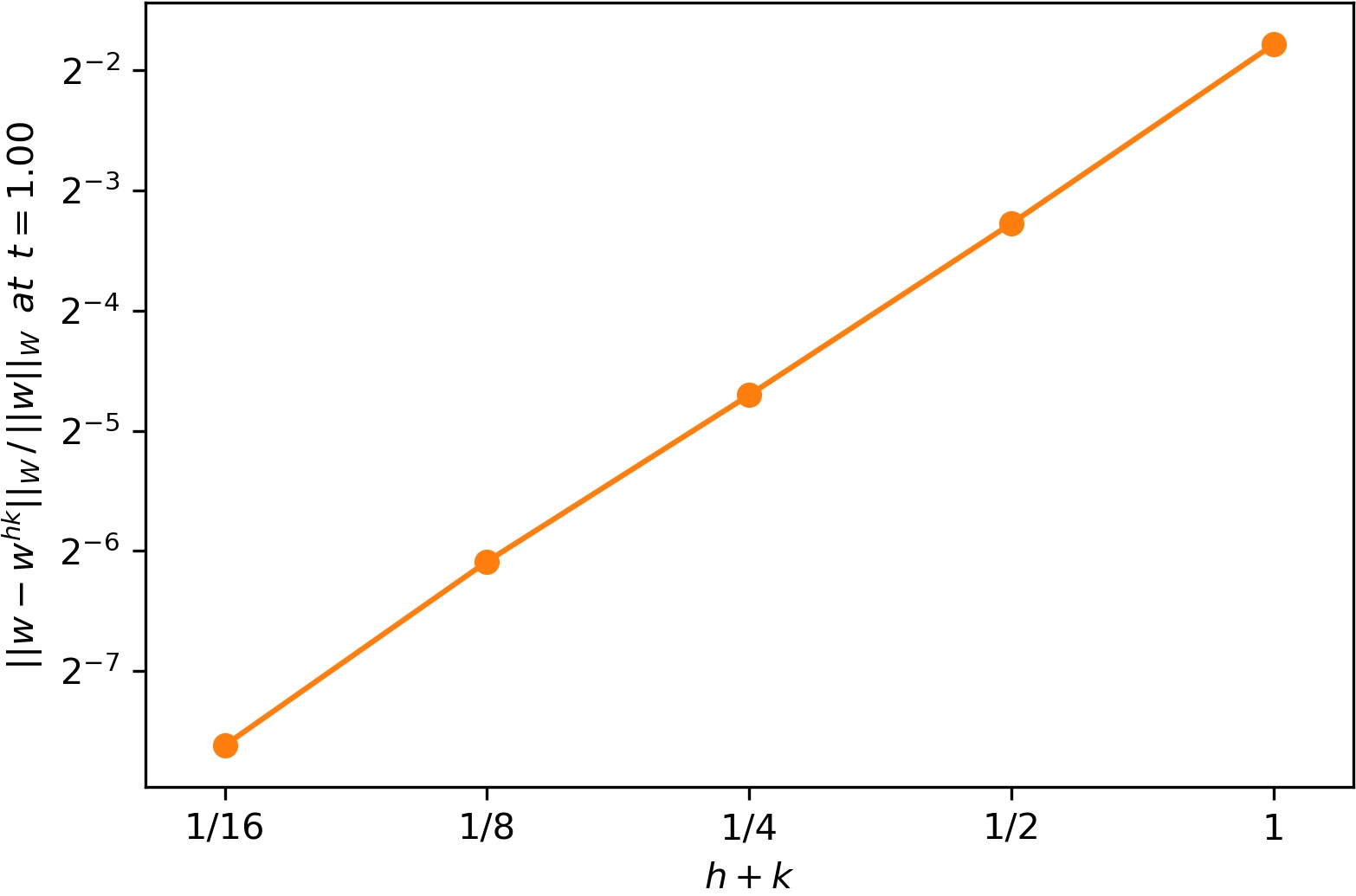}
  \caption{Error estimate $\|w - w^{hk}\|_{W} / \|w\|_{W}$} \label{figSix}
\end{minipage}
\end{figure}

Finally, we explore the numerical convergence orders of the numerical method on the model problem
with $\bm{f}_N(\bm{x},t)=(-0.5,-0.5)$, $\bm{f}_0(\bm{x},t) = (-0.5,-2)$, $\mu(\bm{x}) = 1.0$,
$\kappa(\bm{x})= 0.05$, and $\bm{v^*}(\bm{x},t)= (1,0)$.
We present a comparison of numerical errors $\|\bm{u} - \bm{u}^{hk}\|_V$ and  $\|w - w^{hk}\|_{W}$  
computed for a sequence of solutions to discretized problems. 
We use a uniform discretization of the problem domain and time interval according to the spatial 
discretization parameter $h$ and time step size $k$, respectively. The boundary $\Gamma_C$ of $\Omega$ 
is divided into $1/h$ equal parts. We start with $h = 1/2$ and $k = 1/2$, which are successively halved. 
The numerical solution corresponding to $h = 1/64$ and $k = 1/64$ is taken as the ``exact'' solution 
$\bm{u}$ and $w$ with $\|\bm{u}\|_{V}\doteq 0.19131$ and $\|w\|_{W}\doteq 0.08192$. The results are presented 
in Table \ref{tabOne} and Figures \ref{figFive} and \ref{figSix},
where the dependence of the relative error estimates $\|\bm{u}  - \bm{u}^{hk}\|_V / \|\bm{u}\|_V$ and 
$\|w - w^{hk}\|_{W} / \|w\|_{W}$  with respect to $h+k$ are plotted on a log-log scale. A first order 
convergence is clearly observed for the numerical solutions of the displacement.  The numerical
convergence orders for the numerical solutions of the wear function are somewhat higher than 1.

\medskip
\noindent {\bf Acknowledgments}\\
The project leading to this application has received funding from the European Union's Horizon 2020 research and innovation programme under the Marie Sklo\-do\-wska-Curie grant agreement no.\ 823731 CONMECH.

\end{document}